\documentclass[11pt]{amsart}
\usepackage{amssymb,amsfonts,amsmath,amsopn,amstext,amscd,latexsym,leqno,xy}

\input xy
\xyoption{all}

\theoremstyle{remark}

\newtheorem{para}{\bf}[subsection]
\newtheorem{example}[para]{\bf Example}

\newtheorem{rem}[para]{\bf Remark}
\theoremstyle{definition}

\newtheorem{dfn}[para]{Definition}

\theoremstyle{plain}

\newtheorem{thm}[para]{Theorem}
\newtheorem{lemma}[para]{Lemma}
\newtheorem{cor}[para]{Corollary}
\newtheorem{prop}[para]{Proposition}

\newcommand{\bbA}{{\mathbb A}}

\newcommand{\bbN}{{\mathbb N}}

\newcommand{\bbQ}{{\mathbb Q}}

\newcommand{\bbZ}{{\mathbb Z}}

\newcommand{\frb}{{\mathfrak b}}
\newcommand{\frc}{{\mathfrak c}}

\newcommand{\frg}{{\mathfrak g}}
\newcommand{\frh}{{\mathfrak h}}

\newcommand{\frl}{{\mathfrak l}}

\newcommand{\frn}{{\mathfrak n}}

\newcommand{\frp}{{\mathfrak p}}

\newcommand{\fru}{{\mathfrak u}}

\newcommand{\frx}{{\mathfrak x}}

\newcommand{\frz}{{\mathfrak z}}

\newcommand{\frH}{{\mathfrak H}}

\newcommand{\frX}{{\mathfrak X}}

\newcommand{\cA}{{\mathcal A}}

\newcommand{\cC}{{\mathcal C}}
\newcommand{\cD}{{\mathcal D}}

\newcommand{\cF}{{\mathcal F}}

\newcommand{\cH}{{\mathcal H}}

\newcommand{\cM}{{\mathcal M}}

\newcommand{\cO}{{\mathcal O}}

\newcommand{\End}{{\rm End}}

\newcommand{\hUg}{\hat{U}(\frg)}

\newcommand{\Ug}{U(\frg)}

\newcommand{\hUh}{\hat{U}(\frh)}
\newcommand{\Uh}{U(\frh)}

\newcommand{\hUpI}{\hat{U}(\frp_I)}

\newcommand{\al}{\alpha}
\newcommand{\hSg}{\hat{S}(\frg)}
\newcommand{\Sg}{S(\frg)}
\newcommand{\hSh}{\hat{S}(\frh)}
\newcommand{\Sh}{S(\frh)}

\newcommand{\car}{\stackrel{\cong}{\longrightarrow}}

\newcommand{\hO}{\hat{\cO}}
\newcommand{\hOc}{\hat{\cO}_\chi}\newcommand{\hOcp}{\hat{\cO}_\chi^\frp}
\newcommand{\hM}{\hat{M}}
\newcommand{\hMIl}{\hat{M}_I(\lambda)}
\newcommand{\hMl}{\hat{M}(\lambda)}

\newcommand{\hLl}{\hat{L}(\lambda)}

\newcommand{\hZg}{\hat{Z}(\frg)}

\newcommand{\Mod}{{\rm Mod}}

\begin{document}

\title{Verma modules over $p$-adic Arens-Michael envelopes of reductive
Lie algebras} \author[Tobias Schmidt]{Tobias Schmidt}
\address{Mathematisches Institut\\ Westf\"alische Wilhelms-Universit\"at
M\"unster\\ Einsteinstr. 62\\ D-48149 M\"unster, Germany}
\email{toschmid@math.uni-muenster.de}

\maketitle

\begin{abstract} Let $K$ be a locally compact
$p$-adic field, $\frg$ a split reductive Lie algebra over $K$ and
$U(\frg)$ its universal enveloping algebra. We investigate the
category $\cC_\frg$ of coadmissible modules over the $p$-adic
Arens-Michael envelope $\hat{U}(\frg)$ of $U(\frg)$. Let
$\frp\subseteq\frg$ be a parabolic subalgebra. The main result
gives a canonical equivalence between the classical parabolic
$BGG$ category of $\frg$ relative to $\frp$ and a certain
explicitly given highest weight subcategory of $\cC_\frg$. This
completely clarifies the "Verma module theory" over
$\hat{U}(\frg)$.
\end{abstract}

\normalsize

\tableofcontents

\vskip30pt

\section{Introduction}

Let $p$ be a prime number and let $K$ be a locally compact
$p$-adic field. Let $G$ be a $d$-dimensional $p$-adic Lie group
defined over $K$ with a split reductive Lie algebra $\frg$. Let
$U(\frg)$ denote the universal enveloping algebra of $\frg$.

\vskip8pt

The Arens-Michael envelope $\hat{U}(\frg)$ of $U(\frg)$ equals the
completion of $U(\frg)$ with respect to all submultiplicative
seminorms. Being an interesting $p$-adic power series envelope of
$U(\frg)$ in its own right it is also an important technical tool
in the study of locally analytic $G$-representations (e.g.
\cite{TeitelbaumParis} for a short introduction). As a ring it is
best understood as Fr\'echet-Stein algebra (in the sense of
Schneider-Teitelbaum, cf. \cite{ST5}), i.e. a noncommutative
version of the ring of holomorphic functions on rigid analytic
affine $d$-space. From this angle the coherent module sheaves on
affine $d$-space are generalized to the abelian category
$\cC_\frg$ of {\it coadmissible} (left) $\hat{U}(\frg)$-modules.
Due to the close relation of $\hat{U}(\frg)$ to the locally
analytic distribution algebra of $G$ the latter category is a
first approximation to the category of admissible locally analytic
$G$-representations. Besides the short note \cite{SchmidtSTAB}
there are practically no results about the specific structure of
$\cC_\frg$ so far.

\vskip8pt

In this note we single out certain full subcategories of
$\cC_\frg$ and establish canonical equivalences to certain
well-known highest weight categories over $\frg$. In particular,
this completely clarifies the "Verma module theory" over
$\hat{U}(\frg)$. To be more precise, let $\frp\subseteq\frg$ be a
parabolic subalgebra. On the one hand, we then have the well-known
parabolic $BGG$ category $\cO^\frp$ in the sense of
Bernstein-Gelfand-Gelfand and Rocha-Caridi (\cite{BGG2},
\cite{Rocha-Caridi}). This is a certain full abelian subcategory
of finitely generated $\frg$-modules with appropriate finiteness
conditions for the action of the Levi subalgebra and the nilpotent
radical of $\frp$ respectively. It is known to be artinian and
noetherian and allows a block decomposition with respect to the
central action. Any block is equivalent to a category of finitely
generated modules over a quasi-hereditary finite dimensional
$K$-algebra. The structure of the latter algebras was made
explicit by work of W. Soergel (\cite{Soergel}). Prominent objects
in $\cO^\frp$ are the generalized Verma modules in the sense of J.
Lepowsky (\cite{Lepowsky}).

\vskip8pt

 In defining a genuine $p$-adic counterpart of $\cO^\frp$ over $\hat{U}(\frg)$ we build upon
 a certain weight theory for topological Fr\'echet modules over commutative Fr\'echet algebras (\cite{Feaux}).
 Applying it to the Arens-Michael envelope $\hat{U}(\frh)$ of a Cartan subalgebra $\frh$ of $\frg$ enables us to explicitly
 define certain highest weight categories $\hO^\frp$
 within $\cC_\frg$ by imposing appropriate compactness
 conditions on the weights related to $\frp$. The main objects turn out to be certain
 Verma type modules whose properties closely parallel the classical case.
In particular, they admit unique irreducible quotients
parametrized by the linear dual of $\frh$ and any irreducible
object in $\hO^\frp$ occurs like this. To go further, the
existence of the $p$-adic
 Harish-Chandra homomorphism (\cite{KohlhaaseI}) leads to a decomposition
 of $\hO^\frp$ with respect to central characters $\chi$ of $\hUg$. The blocks
 $\hO^\frp_\chi$ are noetherian and artinian. This makes possible to prove the following main result.

\vskip8pt

As with any Arens-Michael envelope there is a natural map
$\Ug\rightarrow\hUg$. We show that base change along this map
induces an equivalence of categories $\cO^\frp\car\hO^\frp$ which
preserves the central blocks, the Verma modules and their
irreducible quotients. A quasi-inverse can be given explicitly.
The proof of our main result builds on results of
\cite{SchmidtSTAB} and well-known properties of the categories
$\cO^\frp$.

\vskip8pt

We assemble some information on quasi-hereditary algebras and
$BGG$-reciprocity in an appendix.

\vskip8pt

{\it Acknowledgement.} Part of this work was done during a stay of
the author at the IHP Paris during the Galois trimester 2010. The
author thanks this institution for its hospitality and generosity.
The author also wishes to thank M. Strauch for giving three
inspiring lectures at this event on the $BGG$ category $\cO$ and
locally analytic representations. He also thanks B. Schraen for
pointing out some valuable suggestions on an earlier version of
this text.

\section{Diagonalisable modules}\label{sectdiag}


We begin by reviewing a notion of semisimplicity for topological
Fr\'echet-modules (developed by T. F\'eaux deLacroix, cf.
\cite{Feaux}). The exposition is adapted to our purposes. For all
notions of nonarchimedean functional analysis we refer to P.
Schneider's monograph \cite{NFA}.

\vskip8pt

Let $K$ be a locally compact $p$-adic field and $\cH$ a commutative $K$-algebra. Let $\cH^*$ denote
the set of $K$-{\it valued weights} of $\cH$, i.e. the set of
$K$-algebra homomorphisms $\cH\rightarrow K$. A subset
$Y\subseteq\cH^*$ is called {\it relatively compact} if there are
finitely many elements $h_1,...,h_l$ in $\cH$ such that the map
$$Y\longrightarrow K^l, \;\lambda\mapsto
(\lambda(h_1),...,\lambda(h_l))$$ is injective with relatively
compact image. Let $\cM(\cH)$ denote the category whose objects
are $K$-Fr\'echet spaces $M$ endowed with an action of $\cH$ by
continuous $K$-linear endomorphisms. Morphisms are continuous
$K$-linear maps compatible with $\cH$-actions.

\vskip8pt

Let $\lambda\in\cH^*$. Following \cite{Feaux} a nonzero $m\in M$
is called a {\it$\lambda$-weight vector} if $h.m=\lambda(h).m$ for
all $h\in\cH$. In this case $\lambda$ is called a {\it weight of
$M$}. The closure $M_\lambda$ in $M$ of the $K$-vector space
generated by all $\lambda$-weight vectors is called the
$\lambda$-{\it weight space} of $M$. The module $M$ is called
$\cH$-{\it diagonalisable} if there is a set of weights
$\Pi(M)\subseteq\cH^*$ with the property: to every $m\in M$ there
exists a family $\{m_\lambda\in M_\lambda\}_{\lambda\in\Pi(M)}$
converging cofinite against zero in $M$ and satisfying
$$m=\sum_{\lambda\in\Pi(M)}m_\lambda.$$
Given an $\cH$-diagonalisable module $M$ we may form the abstract
$\cH$-module $$M^{ss}=\oplus_{\lambda\in\Pi(M)}M_\lambda$$
(depending on the choice of $\Pi(M)$).

\begin{prop}\label{diag}
Let $M$ be $\cH$-diagonalisable with a relatively compact set of
weights $\Pi(M)$. The following hold:
\begin{itemize}
\item[(i)] Given $m=\sum_{\lambda{\in\Pi(M)}} m_\lambda$ in $M$
the weight components $m_\lambda$ are uniquely determined by $m$.
If $M$ is contained in a closed $\cH$-invariant subspace of $M$
then so are all $m_\lambda$.

 \item[(ii)] $M$ has no other weights besides the set $\Pi(M)$.

\item[(iii)] Suppose additionally that $\dim_K M_\lambda<\infty$
for all $\lambda\in\Pi(M)$. The map $${\rm (*)}\;\;\;N\mapsto
N\cap M^{ss}$$ induces an inclusion preserving bijection between
the $\cH$-invariant closed subspaces of $M$ and the abstract
$\cH$-invariant subspaces of $M^{ss}$. The inverse is given by
passing to the closure in $M$. \item[(iv)] If in the situation of
(iii) $M$ admits additionally an action of a $K$-algebra
$\cH\subseteq\cA$ that stabilizes $M^{ss}$ then the bijection {\rm
(*)} descends to $\cA$-invariant objects.
\end{itemize}
\end{prop}

\begin{proof}
This follows from Satz 1.3.19 and Kor. 1.3.22 of \cite{Feaux}.
Note that $K$-Fr\'echet spaces are Hausdorff, complete and
barrelled.
\end{proof}

We let $\cD(\cH)$ denote the full subcategory of $\cM(\cH)$ whose
objects are $\cH$-diagonalisable modules $M$ over a relatively
compact set of weights $\Pi(M)\subseteq \cH^*$ with finite
dimensional weight spaces $M_\lambda$. By the proposition, given
$M\in\cD(\cH)$ the definition of $M^{ss}$ depends solely on $M$
and coincides with the {\it socle} of the abstract $\cH$-module
$M$. Let $Vec_K$ be the category of abstract $K$-vector spaces.
The following proposition is easily checked.
\begin{prop}\label{exact}
The forgetful functor to $Vec_K$ endowes $\cD(\cH)$ with the
structure of exact category. The latter is stable under passage to
closed $\cH$-invariant subspaces and to the corresponding
quotients. The functor on $\cD(\cH)$
$$M\mapsto M^{ss}$$ into the category of abstract $\cH$-modules is
faithful and exact.
\end{prop}

\section{Highest weight categories}\label{secthyper}
Let $p$ be a prime number. Throughout this section $K$ denotes a
locally compact $p$-adic field. Let $|.|$ be a nonarchimedean
valuation on $K$ with $|p|=p^{-1}$ inducing its topology.
\subsection{Fr\'echet-Stein algebras}\label{subsectFS}
In \cite{ST5} P. Schneider and J. Teitelbaum introduce the notion
of Fr\'echet-Stein algebra and show that locally analytic
distribution algebras of compact $p$-adic Lie groups are of such
type (see remark below). Since Arens-Michael envelopes of Lie
algebras over $K$ are another example of this type (see below) we
briefly review the definition.\footnote{Our definition is adapted
to our purposes and slightly more restrictive than in \cite{ST5}.}
A $K$-Fr\'echet algebra $A$ is called (two-sided) {\it
Fr\'echet-Stein} if there is a sequence $q_1\leq q_2 \leq...$ of
algebra norms on $A$ defining its Fr\'echet topology and such that
for all $m\in\bbN$ the completion $A_m$ of $A$ with respect to
$q_m$ is a left and right noetherian $K$-Banach algebra and a flat
left and right $A_{m+1}$-module via the natural map
$A_{m+1}\rightarrow A_{m}$. Any such algebra $A$ gives rise to a
certain full subcategory $\cC_A$ of all (left) $A$-modules, the
{\it coadmissible modules}. As Fr\'echet-Stein algebras are
typically non-noetherian $\cC_A$ serves as a well-behaved
replacement for the category of all finitely generated (left)
$A$-modules. Instead of giving all details of the construction
(cf. \cite{ST5}, \S3) we summarize some basic properties of
$\cC_A$ in the following proposition.
\begin{prop}\label{coadmissibles}
Let $A$ be a Fr\'echet-Stein algebra.
\begin{itemize}
    \item[(i)] The direct sum of two coadmissible $A$-modules is
    coadmissible.
    \item [(ii)] the (co)kernel and (co)image of an arbitrary
    $A$-linear map between coadmissible $A$-modules is
    coadmissible.
    \item[(iii)] the sum of two coadmissible submodules of a
    coadmissible $A$-module is coadmissible.
    \item[(iv)] any finitely generated submodule of a coadmissible
    $A$-module is coadmissible.
    \item[(v)] any finitely presented $A$-module is coadmissible.
    \item[(vi)] $\cC_A$ is an abelian category.
    \item[(vii)] any coadmissible $A$-module $M$ is
equipped with a canonical Fr\'echet topology making it a
topological $A$-module. Any $A$-linear map between two
coadmissible $A$-modules is continuous and strict with closed
image with respect to canonical topologies.
\end{itemize}
\end{prop}
\begin{proof}\cite{ST5}, Cor. 3.4/3.5 and Lem. 3.6.\end{proof}

Let $A$ be a Fr\'echet-Stein algebra. We will make much use of the
following basic property of the canonical topology.

\begin{lemma}\label{canonicaltop}
For any coadmissible $A$-module $M$ and any abstract $A$-submodule
$N\subseteq M$ the following are equivalent:

\begin{itemize}
    \item[(i)] $N$ is coadmissible.
\item[(ii)] $M/N$ is coadmissible.
    \item[(iii)] $N$ is closed in the canonical topology of $M$.
\end{itemize}
\end{lemma}
\begin{proof}\cite{ST5}, Lem. 3.6.\end{proof}

\begin{rem} Let $G$ denote a locally $K$-analytic
group. With respect to the convolution product the strong dual
$D(G)$ of the $K$-vector space of locally analytic functions on
$G$ is a topological algebra, the so-called {\it locally analytic
distribution algebra} of $G$. The main result of [loc.cit.] proves
that, in case $G$ is compact, $D(G)$ is a Fr\'echet-Stein
algebra. This enables the authors to develop a general theory of
{\it admissible} locally analytic $G$-representations generalizing
the classical notion (\cite{Cartier}) of an {\it admissible}
smooth $G$-representation.

The theory is modelled according to the example
$G=\bbZ_p$, the additive group of $p$-adic integers. In this case, the Fourier
isomorphism of Y. Amice (\cite{Amice2}) identifies $D(G)$ with the ring of
holomorphic functions on the rigid analytic open unit disc. The
latter is a quasi-Stein space in the sense of R. Kiehl
(\cite{Kiehl}).
\end{rem}

\subsection{Arens-Michael envelopes}\label{subsecthyper}

An {\it Arens-Michael $K$-algebra} is a locally convex $K$-algebra
topologically isomorphic to a projective limit of $K$-Banach
algebras. For the theory of such algebras (over the complex
numbers) we refer to the book by A.Y. Helemskii (\cite{Helemskii},
chap. V). Given a locally convex $K$-algebra $A$ its {\it
Arens-Michael envelope} $\hat{A}$ equals the Hausdorff completion
of $A$ with respect to the family of all continuous
submultiplicative seminorms on $A$. It is universal with respect
to continuous $K$-algebra homomorphisms of locally convex
$K$-algebras into Arens-Michael algebras. It comes equipped with a
continuous algebra homomorphism
$$A\longrightarrow \hat{A}$$ with dense image. This construction
gives a functor

$$A\mapsto \hat{A}$$

from locally convex $K$-algebras to Arens-Michael algebras which
is compatible with projective tensor products and passage to
quotients by twosided ideals (cf. \cite{Pirkovskii}, 6.1 for the
complex case; the proofs generalize).

\vskip8pt

Let $\frg$ be a Lie algebra over $K$ of dimension $d$ and let
$U(\frg)$ be its universal enveloping algebra endowed with the
finest locally convex topology. Let $\hUg$ be its Arens-Michael
envelope. It will be convenient to realize $\hUg$ in the following
explicit way.
Fix a $K$-basis $\frx_1,....,\frx_d$ of $\frg$. Using the
associated PBW-basis for $U(\frg)$ we may define for each $r>0$ a
vector space norm on $U(\frg)$ via
\begin{equation}\label{norms} ||\sum_\al
d_\al\frX^\al||_{\frX,r}=\sup_\al |d_\al|r^{|\al|}\end{equation}
where
$\frX^\al:=\frx_1^{\al_1}\cdot\cdot\cdot\frx_d^{\al_d},~\al\in\bbN_0^d$
and $|\al|:=\al_1+\cdot\cdot\cdot +\al_d$.
\begin{prop}\label{AMexplicit}
The Hausdorff
completion of $U(\frg)$ with respect to the family of norms
$||.||_{\frX,r},~r>1$ is an Arens-Michael algebra. The canonical homomorphism from $\hat{U}(\frg)$ into it
is a topological algebra isomorphism.
\end{prop}
\begin{proof}
It is easy to see that each norm $||.||_{\frX,r},~r>1$ is
submultiplicative and that this family is cofinal in the directed
set of all submultiplicative seminorms on $U(\frg)$ (cf.
\cite{SchmidtSTAB})
\end{proof}
\begin{rem}
In analogy to the complex hyperenveloping algebra introduced by
P.K. Rasevskii (cf. \cite{Rasevskii}) the completion of $U(\frg)$
with respect to the norms $||.||_{\frX,r},~r>1$ is sometimes
called the {\it $p$-adic hyperenveloping algebra} of
$\mathfrak{g}$ (\cite{SchmidtSTAB},\cite{TeitelbaumParis}).
\end{rem}
The above discussion shows that we have a functor
$$\frg\mapsto \hUg$$ from finite dimensional Lie algebras over $K$ to
Arens-Michael $K$-algebras satisfying the obvious compatibilities
with respect to products/projective tensor products and passage to
quotients.
It is immediate that everything we said above may be applied {\it
mutatis mutandis} to the symmetric algebra $S(\frg)$ of $\frg$.
\begin{prop}\label{comparison}
The algebras $\hUg$ and $\hSg$ are Fr\'echet-Stein algebras which
are integral domains. The $K$-linear isomorphism $\Ug\simeq\Sg$
induced by the choice of basis $\frx_1,...,\frx_d$ extends to a
topological isomorphism $\hUg\simeq\hSg$.
\end{prop}
\begin{proof}
\cite{SchmidtVECT}, Thm. 2.3 and \cite{SchmidtSTAB}, Thm. 2.1. The
corresponding noetherian Banach algebras arise as the completions
with respect to single norms $||.||_{\frX,r}$.
\end{proof}

For future reference we denote the completion of $U(\frg)$ with
respect to the norm $||.||_{\frX,r}$ by $U_r(\frg)$. It is a
noetherian Banach algebra and the natural map
$\hat{U}(\frg)\rightarrow U_r(\frg)$ is flat (\cite{ST5}, Remark
3.2).

\vskip8pt

Let $V$ and $W$ be two locally convex $K$-spaces. We denote the
completed projective tensor product of $V$ and $W$ over $K$ by
$V\hat{\otimes}_KW$.
\begin{lemma}\label{PBW}
Suppose $\frg_1,...\frg_n$ are Lie subalgebras of $\frg$ such that
$\frg_1\oplus\cdot\cdot\cdot\oplus\frg_n=\frg$ as $K$-vector
spaces. There exists a unique isomorphism $f$ of topological
bimodules
$$\hat{U}(\frg_1)\hat{\otimes}_K\cdot\cdot\cdot\hat{\otimes}_K\hat{U}(\frg_n)\car
\hat{U}(\frg)$$ such that
$f(u_1\hat{\otimes}\cdot\cdot\cdot\hat{\otimes}u_n)=u_1\cdot\cdot\cdot
u_n$ for $u_i\in \hat{U}(\frg_i)$. Similarly for $S$ instead of
$U$.
\end{lemma}

\begin{proof}
We have the usual PBW-isomorphism of bimodules
$$U(\frg_1)\otimes_K\cdot\cdot\cdot\otimes_K U(\frg_n)\car
U(\frg)$$ (\cite{Dixmier}, Prop. 2.2.10) and similarly for $S$. In
the case of $S$ the latter is even an isomorphism of $K$-algebras
and compatibility of Arens-Michael envelopes with projective
tensor products yields the claim. The second claim of Prop.
\ref{comparison} applied to all algebras $\frg_1,...,\frg_n$ and
$\frg$ then yields the claim for $U$.
\end{proof}

\vskip8pt

We conclude this paragraph with some remarks in case $\frg$ is
abelian. By the universal property of the Arens-Michael envelope
any weight $\hUg\rightarrow K$ (cf. sect. \ref{sectdiag}) is
automatically continuous. The map
$$\hUg^{*}\car\frg^*, \lambda\mapsto [\frx\mapsto
\lambda(\frx)]$$ therefore identifies the set $\hUg^*$ canonically
with the $K$-linear dual $\frg^*$ of $\frg$. This identification
is compatible
with the isomorphism of locally convex $K$-algebras (cf. prop.
\ref{comparison})
\begin{equation}\label{basis}\frX: \hUg\car\cO(\bbA^{d,an}_K)\end{equation} mapping a chosen Lie
algebra basis $\frX:=\{\frx_1,...\frx_d\}$ to a system of
coordinates on $\bbA^{d,an}_K$. Here, $\bbA^{d,an}_K$ denotes the
rigid analytic affine $d$-space over $K$ (\cite{BGR}, 9.3.4) and
$\cO(\bbA^{d,an}_K)$ equals its ring of rigid analytic functions
viewed as a locally convex algebra in the usual way.

\subsection{Reductive Lie algebras}\label{subsectsemisimple}

From now on $\frg$ is a split reductive Lie algebra over K. We
refer to \cite{Dixmier} for the basic structure of such algebras.
Let $\frh$ be a Cartan subalgebra of $\frg$, $\frb$ be a Borel
subalgebra containing $\frh$, $\Phi$ the root system of $\frg$
with respect to $\frh$, and $\Phi^+$ and $\Delta$ the set of
positive and simple roots, respectively. Let $W$ denote the Weyl
group of $\Phi$. Denote by $\frn$ and $\frn^-$ the nilpotent
radicals of $\frb$ and $\frb^-$ respectively. We have
$\frn=[\frb,\frb]$ and $\frh\simeq\frb/\frn$ canonically. Let
$\frh^*$ denote the $K$-linear dual and put $l:=\dim_K\frh$. For
each root $\al\in\Phi$ let $\frg_\al$ be the one dimensional root
space in $\frg$. Finally, we let $\Lambda_r\subseteq\Lambda$ be
the root lattice and the integral weight lattice respectively.
$\Lambda$ contains the subsemigroup of dominant integral weights
$\Lambda^+$.

\vskip8pt

In the following we will fix a Chevalley basis
$\{\frx_\beta,\beta\in\Phi, h_\alpha, \alpha\in\Delta\}$ of the
derived algebra $\frg'=[\frg,\frg]$. We fix once and for all a
$K$-basis for the center $\frc$ of $\frg$ which, together with
$\{h_\alpha\}_{\alpha\in\Delta}$ gives rise to a $K$-basis of
$\frh$. Throughout this work we will work with this fixed choice
of $K$-basis of $\frh$ and call it $\frH$.

\vskip8pt

\subsection{Generalized Verma modules}

Generalized Verma modules (GVM) for parabolic subalgebras of
reductive Lie algebras were first introduced by J. Lepowsky
(\cite{Lepowsky}). For an extensive treatment of such modules we
refer to V. Mazorchuk's monograph (\cite{Mazorchuk}).

Let $\frp_I$ be a parabolic subalgebra of $\frg$ containing $\frb$
and let $I\subseteq\Delta$ be the associated subset of simple
roots. Let $\Phi_I\subseteq\Phi$ be the corresponding root system
with positive roots $\Phi_I^+$, negative roots $\Phi_I^-$ and Weyl
group $W_I\subseteq W$. Let $$\frp_I=\frl_I\oplus\fru_I$$ be a
Levi decomposition of $\frp_I$ with Levi subalgebra $\frl_I$ and
nilpotent radical $\fru_I\subseteq\frb$. Since $\frl_I$ is
reductive there is a further decomposition
$$\frl_I=\frg_I\oplus\frz_I$$
where $\frg_I$ and $\frz_I$ denote the derived algebra and the
center of $\frl_I$ respectively.

\vskip8pt

Setting $\frh_I:=\oplus_{\alpha\in I} Kh_\alpha$ defines a Cartan
subalgebra of the semisimple algebra $\frg_I$ such that
$$\frh=\frh_I\oplus\frz_I.$$
This induces a decomposition $\frh^*=\frh_I^*\oplus\frz_I^*$. For
each $\lambda\in\frh^*$ denote by $\lambda_I$ its projection to
$\frh_I^*$. Finally, let $\al^\vee$ denote the dual root to a
given $\al\in\Phi$ and define
$$\Lambda_I^+:=\{\lambda\in\frh^*| \langle
\lambda,\alpha^\vee\rangle\in\bbZ_{\geq 0} {\rm~for~all~}\alpha\in
I\}.$$

\vskip8pt

Let $\lambda\in\Lambda_I^+$. Denote by $L_I(\lambda)$ the
irreducible finite dimensional $\frg_I$-module with highest weight
$\lambda_I$ viewed as $\frl_I$-module (by letting $\frz_I$ act
through the projection of $\lambda$ to $\frz_I^*$). The
generalized Verma module (GVM) of weight $\lambda$ relative to
$\frp_I$ is the left $U(\frg)$-module
$$M_I(\lambda):=U(\frg)\otimes_{U(\frp_I)}L_I(\lambda)$$
where $L_I(\lambda)$ is inflated to a $\frp_I$-module via the
projection $\frp_I\rightarrow\frp_I/\fru_I=\frl_I$.

\vskip8pt

In case $I=\emptyset$ let $M(\lambda), N(\lambda), L(\lambda)$ be the classical Verma
module of weight $\lambda$, its unique maximal submodule and its
unique irreducible quotient respectively. The module
$M_I(\lambda)$ is generated by a maximal vector $m^+\in
M_I(\lambda)$ (i.e. $\frn.m^+=0$) and, thus, admits a surjection
$M(\lambda)\rightarrow M_I(\lambda)$. Hence $L(\lambda)$ equals
the unique irreducible quotient of $M_I(\lambda)$.

\subsection{The parabolic BGG category}\label{parabolicBGG}

We keep the notation of the preceding subsections and fix a
parabolic subalgebra $\frp=\frp_I$ of $\frg$. Let $\Mod (U(\frg))$
denote the category of all left $U(\frg)$-modules. Let us recall
the parabolic $BGG$ category \footnote{Traditionally, the
categories $\cO^\frp$ are defined for semi-simple (complex) Lie
algebras. Following \cite{OrlikStrauchJH} we extend their
definition here to general reductive Lie algebras over $K$.} in
the sense of A. Rocha-Caridi (\cite{Rocha-Caridi}). It equals the
full subcategory of $\Mod(U(\frg))$ consisting of modules $M$ such
that
\begin{itemize}
    \item[(i)] $M$ is finitely generated as $U(\frg)$-module;
    \item[(ii)] viewed as  a $U(\frl_I)$-module, $M$ is the direct
    sum of finite dimensional simple modules;
    \item[(iii)] $M$ is locally $\fru_I$-finite.
\end{itemize}
Here, (iii) means that $U(\fru_I)m$ is finite dimensional for each
$m\in M$.

\vskip8pt

It is known that $\cO^\frp$ is a $K$-linear, abelian, artinian and
noetherian category which is closed under submodules and quotients
(\cite{HumphreysBGG}). There are two extreme cases: the case
$I=\emptyset$ recovers the classical category $\cO$ in the sense
of Bernstein-Gelfand-Gelfand (\cite{BGG1}), while $I=\Delta$
yields the (semisimple) category of finite dimensional
$U(\frg)$-modules. Obviously, $\frp\subset\frp'$ implies
$\cO^{\frp'}\subset\cO^\frp$. We summarize a few more properties
of $\cO^\frp$ in the following theorem. Let $Z(\frg)$ denote the
center of $U(\frg)$ and $\chi$ a central character
$Z(\frg)\rightarrow K$.

\begin{thm}\label{propertiesparabolic}

\begin{itemize}
    \item[(i)] The modules $M_I(\lambda), \lambda\in\Lambda_I^+$ belong
    to $\cO^\frp$;
    \item[(ii)] the modules $L_I(\lambda), \lambda\in\Lambda_I^+$
    exhaust the set of irreducible objects in $\cO^\frp$;
    \item[(iii)] $\cO^\frp=\oplus_\chi \cO^\frp_\chi$ where
    $\cO^\frp_\chi$ consists of modules $M_\chi$ such that
    $(\ker\chi)^{n(m)}.m=0$ for some $n(m)\geq 1$ and all $m\in
    M_\chi$;
    \item[(iv)] $\cO^\frp$ has enough projective objects;
    \item[(v)] there is a bijection between irreducible objects and
    indecomposable projective objects in $\cO^\frp$;
    \item[(vi)] each $\cO^\frp_\chi$ is (noncanonically) equivalent to a
    category of finitely generated right modules over a
    finite-dimensional $K$-algebra $A^\frp_\chi$.
\end{itemize}
\end{thm}
\begin{proof}
The proofs mentioned in \cite{Mazorchuk}, \S5 all generalize
immediately to our setting of a split reductive $K$-algebra
$\frg$. Note that the $K$-rationality of the block decomposition
in (iii) follows from the (generalized) Harish-Chandra map
[loc.cit.], \S4.3 (compare also the argument in the proof of Prop.
\ref{block} below) like this. Let $M\in\cO^\frp$. Viewed as a
$\frl_I$-module $M$ decomposes into the direct sum over isotypic
components $M_\lambda$ where $\lambda$ equals an isomorphism class
of finite dimensional simple $\frl_I$-modules. Since $\frg$ is
split the highest weight of any such finite-dimensional simple
$\frl_I$-module is an element, denoted $\lambda$ again, of the
$K$-linear dual of $\frh$. If we decompose $\lambda$ into a sum of
two linear forms on the Cartan subalgebra of $\mathfrak{g}_I$ and
the center $\frz_I$ of $\frl_I$ respectively we obtain a character
of the tensor product of the algebras $Z(\mathfrak{g}_I)$ and
$S(\frz_I)$. Here, $S(\cdot)$ refers as usual to the symmetric
algebra. Let $K'$ be a finite field extension of $K$ and let $m$
be an element of $K'\otimes_K M_\lambda$ on which $Z(\frg)$ acts
through a $K'$-valued character $\chi$. By the very construction
of the generalized Harish-Chandra map
$$\varphi_I: Z(\frg)\longrightarrow Z(\mathfrak{g}_I)\otimes_K
S(\frz_I)$$ we then have $\chi=\lambda\circ\varphi_I$ and, thus,
$\chi$ takes values in $K$.
\end{proof}

The category $\cO^\frp_\chi$ satisfies an analogue of the
classical BGG reciprocity principle (\cite{BGG2}) or,
equivalently, the algebra $A^\frp_\chi$ appearing in (v) is a
so-called {\it BGG algebra}. For more information in this
direction and on the related class of {\it quasi-hereditary
algebras} we refer to the appendix.

\vskip8pt

Let $\Gamma_I\subseteq\frh^*$ be the positive cone over the set
$\Phi^+\setminus\Phi^+_I$, i.e. $$\Gamma_I:= \bbZ_{\geq
0}(\Phi^+\setminus\Phi^+_I).$$ Put $\Gamma=\Gamma_\emptyset$.
For $\lambda,\mu\in\frh^*$ we define a partial order on $\frh^*$
as usual via $$\lambda\geq\mu$$ if $\lambda-\mu\in\Gamma$. It will
be convenient to have the following weight characterization of
$\cO^\frp$ as a subcategory of $\cO$.

\begin{lemma}\label{char}
Let $M\in\cO$. Then $M\in\cO^\frp$ if and only if the
$\frh$-weights of $M$ lie in a finite union of cosets of the form
$\lambda-\Gamma_I$.
\end{lemma}
\begin{proof}
Put $\frn^-_I:=\oplus_{\alpha\in\Phi^-_I}\frg_\alpha$ so that
$\frn^-=\frn^-_I\oplus\fru^-_I$. According to \cite{HumphreysBGG}, Prop. 9.3 we are reduced to show that $M$ satisfies
the above condition on weights if and only if it is locally
$\frn^-_I$-finite.

In any case, $M\in\cO$ is finitely generated by
$\frh$-weight vectors $m_1,...,m_s$. Since any
$U(\frg)m_i\subseteq M$ is $\frh$-semisimple it suffices to treat
the case $s=1$. From the decompositions
$$\frg=\fru^-_I\oplus\frp_I=\fru^-_I\oplus\frl_I\oplus\fru_I=\fru^-_I\oplus\frg_I\oplus\frz_I\oplus\fru_I$$
and $$\frg_I=\frn_I^-\oplus\frh_I\oplus\frn_I$$ with
$\frh=\frh_I\oplus\frz_I$ and $\frh_I$ the Cartan subalgebra of
the semisimple algebra $\frg_I$ we obtain

$$\frg=\fru^-_I\oplus\frh\oplus\frn_I^-\oplus\frn$$ and therefore
$$U(\frg)=U(\fru^-_I)\otimes_KU(\frh)\otimes_KU(\frn_I^-)\otimes_KU(\frn)$$
as bimodules. Now if $M\in\cO$ is additionally locally
$\frn^-_I$-finite then multiplying $m$ with
$U(\frn_I^-)\otimes_KU(\frn)$ produces a finite-dimensional
$\frh$-stable subspace generated by finitely many $\frh$-weight
vectors of weights $\lambda$. Multiplying these with $U(\fru^-_I)$
produces only weights of the form $\lambda-\Gamma_I$.

Conversely, let $M$ satisfy the assumption on weights. If
$\lambda$ denotes the weight of $m$, multiplying $m$ by elements
in $(\frn_I^-)^n, n> 0$ produces weights of the form
$\lambda-\beta$ with $\beta\in\bbZ_{\geq 0}\Phi^+_I$. By
assumption only finitely many of such weights can occur whence
$(\frn_I^-)^n.m=0$ for some $n>0.$ Hence $\dim_K
U(\frn_I^-)m<\infty$ and $M$ is locally $\frn_I^-$-finite.

\end{proof}

\subsection{A $p$-adic counterpart}\label{BGG}
We keep the notation developed so far. Recall that $\hUg$ is
Fr\'echet-Stein and denote the category of coadmissible
$\hUg$-modules by $\cC_\frg$ (and similarly for appropriate
subalgebras of $\frg$). By functoriality we have a continuous
homomorphism $\hUh\rightarrow\hUg$ extending the inclusion
$\frh\subset\frg$. We apply the notions of section \ref{sectdiag}
to the commutative $K$-algebra $\hUh$ and the set of elements in
$\frH$. Restriction of scalars via $\hUh\rightarrow\hUg$ induces a
faithful and exact functor
$$\cC_\frg\longrightarrow\cM(\hUh).$$

\begin{lemma}\label{compact}
Let $M\in\cM(\hUh)$ be $\hUh$-diagonalisable with a set of weights
$\Pi(M)\subseteq\frh^*$ contained in finitely many cosets of the
form $\lambda-\Gamma_I$. Then $\Pi(M)$ is relatively compact.
\end{lemma}
\begin{proof}
Invoking the basis elements $\frH=\{h_1,...,h_l\}$ the map
$$\iota: \Pi(M)\rightarrow K^l, \lambda\mapsto
(\lambda(h_1),...,\lambda(h_l))$$ is injective. It suffices to see
that $\iota(\Gamma_I)$ is relatively compact. Since each root in
$\Phi$ is trivial on $\frc$ this image lies in the closed subspace
$K^{|\Delta|}\subseteq K^l$. Since each $h_\alpha,\alpha\in\Delta$
is part of a Chevalley basis we have $h_\alpha(\beta)\in\bbZ$ for
all $\beta\in\Phi$. It follows that the closure of
$\iota(\Gamma_I)$ lies in the closure of $\bbZ^{|\Delta|}$, i.e.
in $\bbZ_p^{|\Delta|}$.
\end{proof}
This lemma enables us to single out the following subcategory of
$\cC_\frg$. Let $\frp=\frp_I$.
\begin{dfn}\label{def} The category $\hat{\cO}^\frp$ for $\hUg$ equals the
full subcategory of $\cC_\frg$ consisting of coadmissible modules
$M$ satisfying:
\begin{itemize}
    \item[(1)] $M$ is $\hUh$-diagonalisable with $\Pi(M)$
    contained in the union of finitely many cosets of the form
    $\lambda-\Gamma_I, \lambda\in\frh^*$.
    \item[(2)] All weight spaces $M_\lambda, \lambda\in\Pi(M)$ are finite
    dimensional over $K$.
\end{itemize}
\end{dfn}
We let $\hO:=\hO^\frb$ in case $I=\emptyset$. Obviously
$\frp\subset\frp'$ implies $\hO^{\frp'}\subset\hO^{\frp}$. Before
we exhibit a class of interesting objects in $\hO^\frp$ we list
some basic formal properties.
\begin{prop}\label{propO}
\begin{itemize}
    \item[(i)] The direct sum in $\cC_\frg$ of two objects of $\hO^\frp$ is in $\hO^\frp$
    \item [(ii)] the (co)kernel and (co)image of an arbitrary
    $\hUg$-linear map between objects in $\hO^\frp$ is in $\hO^\frp$
    \item[(iii)] the sum of two coadmissible submodules of an object in $\hO^\frp$ is in $\hO^\frp$
    \item[(iv)] any finitely generated submodule of an object in
    $\hO^\frp$ is in $\hO^\frp$
    \item[(v)] $\hO^\frp$ is an abelian category.
\end{itemize}
\end{prop}
\begin{proof} This follows from Prop. \ref{coadmissibles} and Lem. \ref{exact}.\end{proof}
\begin{lemma}\label{canonicaltop2}
For any object $M$ in $\hO^\frp$ and any abstract $\hUg$-submodule
$N\subseteq M$ the following are equivalent:

\begin{itemize}
    \item[(i)] $N\in\hO^\frp$
\item[(ii)] $M/N\in\hO^\frp$
    \item[(iii)] $N$ is closed in the canonical topology of $M$.
\end{itemize}
\end{lemma}
\begin{proof}
Lem. \ref{canonicaltop}.
\end{proof}
Recall the exact category $\cD(\hUh)$ of section \ref{sectdiag}.
\begin{lemma}\label{propobjO}
Let $M\in\hO^\frp$. The map $N\mapsto N\cap M^{ss}$ defines an
inclusion preserving bijection between subobjects of
$M\in\cD(\hUh)$ and abstract $U(\frh)$-submodules of $M^{ss}$. It
descends to a bijection between subobjects of $M\in\hO^\frp$ and
abstract $U(\frg)$-submodules of $M^{ss}$.
\end{lemma}
\begin{proof}
$M$ is $U(\frh)$-diagonalisable with set of weights $\Pi(M)$ and
finite dimensional weight spaces. The first statement follows thus
from propositions \ref{diag} and \ref{exact}. For the second
statement observe that the $K$-subalgebra $\cA$ of $\hUg$
generated by $\hUh$ and $U(\frg)$ stabilizes $M^{ss}$ (e.g.
\cite{Dixmier}, Prop. 7.1.2). Again by Prop. \ref{diag} the
bijection descends to closed $\cA$-invariant subobjects of
$M\in\cD(\hUh)$ and abstract $U(\frg)$-submodules of $M^{ss}$. The
$\cA$-action on such a subobject $N\subseteq M$ uniquely extends
to $\hUg$ making $N$ a subobject of $M\in\hO^\frp$ according Lem.
\ref{canonicaltop2}.
\end{proof}
\begin{example}\label{finite0}
Let $M$ be a finite dimensional $\frg$-module. Since the
endomorphism algebra $\End_K(M)$ has a natural and unique
$K$-Banach topology the $\Ug$-action uniquely extends to $\hUg$
yielding $M\in\cC_\frg$. By standard highest weight theory for
reductive Lie algebras (\cite{Dixmier}) $M$ is
$\hUh$-diagonalisable with a finite set of weights contained in
$\Lambda^+$. Hence $M\in\hO^\frp$ and we have an exact and fully
faithful embedding from the finite dimensional $\frg$-modules into
$\hO^\frp$.

\end{example}


\subsection{$p$-adic Verma modules}

We exhibit Verma type modules in $\hO^\frp$. Let
$\lambda\in\Lambda_I^+$. Consider the finite dimensional
irreducible $\frl_I$-module $L_I(\lambda)$. As explained in the
example above the $\frl_I$-action extends to $\hat{U}(\frl_I)$.
Invoking the map $\hat{U}(\frp_I)\rightarrow\hat{U}(\frl_I)$ we
may form the left $\hat{U}(\frg)$-module

$$\hMIl:=\hUg\otimes_{\hUpI} L_I(\lambda).$$

\begin{prop}\label{verma} The module $\hMIl$ lies in $\hO^\frp$ and
we have $\hMIl^{ss}=M_I(\lambda)$. There is a canonical inclusion
preserving bijection between subobjects of $\hMIl$ and abstract
$U(\frg)$-submodules of $M_I(\lambda)$. In particular, the
topological $\hUg$-module $\hMIl$ is topologically irreducible if
and only if the abstract $U(\frg)$-module $M_I(\lambda)$ is
irreducible.
\end{prop}
\begin{proof}
We first show that $\hMIl$ is coadmissible. The left module
$M=\hUg\otimes_K L_I(\lambda)$ is coadmissible being isomorphic to
a direct sum over finitely many copies of $\hUg$. Its submodule
$N$ generated by the elements $x\otimes 1-1\otimes x$ where $x$
runs through a $K$-basis of $\frp_I$ is coadmissible. Being closed
it contains all elements of the form $y\otimes 1-1\otimes y$ with
$y\in\hat{U}(\frp_I)$ whence $M/N$ coincides with
$\hM_I(\lambda)$. Hence $\hMIl$ is coadmissible and its canonical
topology arises as a quotient topology from $M$. In particular,
the natural map $$\hMIl\car\hUg\hat{\otimes}_{\hUpI}
L_I(\lambda)$$ into the completed projective tensor product of the
locally convex $U(\frp_I)$-modules $\hat{U}(\frg)$ and
$L_I(\lambda)$ is a topological isomorphism. Let now
$\fru_I^-=\oplus_{\Phi^-\setminus\Phi_I^-}\frg_\alpha$ so that
$\fru_I^-\oplus\frp_I=\frg.$ Applying Lem. \ref{PBW} to the latter
decomposition and recalling that the completed projective tensor
product is associative we obtain that
$$\hM_I(\lambda)=\hat{U}(\fru_I^-)\otimes_K L_I(\lambda)$$
as left $\hat{U}(\fru_I^-)$-modules. Contemplating the
$\hat{U}(\frh)$-action on this representation we see that
$\hat{M}_I(\lambda)\in\hO^\frp$ and that
$\hat{M}_I(\lambda)^{ss}=U(\fru_I^-)\otimes_K
L_I(\lambda)=M_I(\lambda)$. The final two statements follow now
from Lem. \ref{propobjO}.
\end{proof}


We recall at this point that the irreducibility properties of
generalized Verma modules are dependent-at least in the case of
regular weights- on antidominance properties of the inducing
character. To be more precise, let
$\rho:=1/2\sum_{\al\in\Phi^+}\al$ and consider the following
condition on a weight $\lambda\in\Lambda_I^+$:
$$(*)\;\;\;\; \langle \lambda+\rho,\beta^\vee\rangle\notin\bbZ_{>0}
{\rm~for~all~}\beta\in\Phi^+\setminus\Phi_I.$$

In case $I=\emptyset$ we recover the usual definition of
antidominance (\cite{Dixmier}). Put
$w\cdot\lambda=w(\lambda+\rho)-\rho, w\in W$ for the usual {\it
dot-action} of $W$ on $\frh^*$. If the stabilizer with respect to
this action of $\lambda$ is trivial we call $\lambda$ {\it
regular}.
\begin{thm} Let $\lambda\in\Lambda_I^+.$
\begin{itemize}
    \item[(i)] If $\lambda$ satisfies (*), then $M_I(\lambda)$ is
    irreducible.
    \item[(ii)] If $\lambda$ is regular and $M_I(\lambda)$
    is irreducible, then $\lambda$ satisfies (*).
\end{itemize}
\end{thm}
The preceding theorem is due to Wallach-Conze-Berline-Duflo-Jantzen for which we refer to \cite{HumphreysBGG}, Thm. 9.12. In
case $I=\emptyset$ it holds true without the regularity condition
in (ii) and is due to Bernstein-Gelfand-
Gelfand for which we refer to \cite{Dixmier}, Thm. 7.6.24. We also mention that there
is a much deeper irreduciblity criterion for generalized Verma
modules which avoids any regularity conditions and is due to H.C.
Jantzen. We refer to \cite{HumphreysBGG}, 9.13 for a concise
account.




\subsection{Highest weight modules}

In the following it will be convenient-at least in case
$I=\emptyset$ -to introduce in our setting the notion of a highest
weight module.

So let $\frp=\frb$ be the Borel subalgebra. Let $M$ be for a
moment an arbitrary coadmissible $\hUg$-module. As usual, a {\it
maximal vector} of weight $\lambda\in\frh^*$ in $M$ is a nonzero
element $m\in M_\lambda$ such that $\frn.m=0$. We call a
coadmissible module $M$ a {\it highest weight module with highest
weight} $\lambda$ if it is a cyclic $\hUg$-module on a maximal
vector in $M_\lambda$.
\begin{rem}\label{maximal}
It follows directly from the definition of $\hO$ that any
$M\in\hO$ has a maximal vector. In particular, any irreducible
object in $\hO$ is a highest weight module according to property
$(iv)$ of Prop. \ref{propO}.
\end{rem}
\begin{lemma}\label{vermaII} The coadmissible module $\hMl$ is a highest weight module of weight
$\lambda$.
\end{lemma}
\begin{proof}
This follows as a special case from (the proof of) Prop.
\ref{verma}.



\end{proof}
\begin{prop}\label{highestweightmodules}
Let $M\in\cC_\frg$ be a highest weight module on a maximal vector
$m\in M$ of weight $\lambda\in\frh^*$. We have the
following:\begin{itemize}
    \item[(a)] $M$ is $\hUh$-diagonalisable with a compact set of weights
    $\Pi(M)$ satisfying $\mu\leq\lambda$ for $\mu\in\Pi(M)$.

    \item[(b)] One has $dim_K M_\mu<\infty$ and $dim_KM_\lambda=1$ for all
    $\mu\in\Pi(M)$. In particular, $M\in\hO$ and $M$ is a finite
    length object in $\hO$.
    \item[(c)] Each nonzero quotient of $M$ by a coadmissible
    submodule is again a highest weight module.
    \item[(d)] Each coadmissible submodule of $M$ generated by a
    maximal vector $m\in M$ of weight $\mu<\lambda$ is proper. In
    particular, if $M$ is a simple object then all its maximal vectors
    lie in $K.m$ and hence $End_{\hUg}(M)=K$.
    \item[(e)] $M$ has a unique maximal subobject and a unique
    simple quotient object and, hence, is indecomposable in
    $\cC_\frg$.
    \item[(f)] Let $M,N$ be two highest weight modules of weights $\lambda$ and $\mu$ respectively. We have $dim_K Hom_{\hUg}(M,N)<\infty$.
    If $\lambda\neq\mu$ then $M$ and $N$ are nonisomorphic. If $M$ and $N$ are simple objects and $\lambda=\mu$ then $M\simeq
    N$.
\end{itemize}
\end{prop}
\begin{proof}
Since $M$ is a quotient object of $\hMl$ in $\cC_{\frg}$ we obtain
an $U(\frg)$-linear surjection
$$M(\lambda)=\hMl^{ss}\longrightarrow M^{ss}$$
by right exactness of $(.)^{ss}$. In particular, $M^{ss}$ is a
highest weight module of weight $\lambda$ in $\cO$. All properties
follow then from classical results on highest weight modules in
$\cO$ (e.g. \cite{HumphreysBGG}, Thm. 1.2).
\end{proof}
Let $\hLl$ denote the unique simple quotient of $\hMl$ so that
\begin{equation}\label{simpless}\hLl^{ss}\simeq L(\lambda).\end{equation}
\begin{cor}\label{simples}
The map $\lambda\mapsto [\hLl]$ is a bijection from $\frh^*$ onto
the set of isomorphism classes of irreducible objects of $\hO$.
\end{cor}
\begin{proof}
This follows as in the classical case of category $\cO$ using
\ref{simpless}.
\end{proof}

\section{Block decomposition and the main result}
\subsection{$p$-adic Harish-Chandra homomorphism}
We begin by recalling some standard results on the center
$Z(\frg)$ of $U(\frg)$ (\cite{Dixmier}). Recall that the usual
adjoint action of $\frg$ on itself extends to an action of $\frg$
by derivations on $U(\frg)$ and $S(\frg)$. Let $U(\frg)^\frg$ and
$S(\frg)^{\frg}$ denote the $K$-algebras of invariants.


\vskip8pt

Let $\gamma^\sharp$ be the algebra automorphism of $S(\frh)$
sending a polynomial function $f$ on $\frh^*$ to the function
$\lambda\mapsto f(\lambda-\rho)$. Let $U(\frg)_0$ be the commutant
of $\frh$ in $U(\frg)$. Then $$I:=\Ug\frn^+\cap
\Ug_0=\frn^-\Ug\cap\Ug_0$$ is a two-sided ideal in $\Ug_0$ such
that $\Ug_0=\Uh\oplus I$. The corresponding algebra surjection
$\varphi: \Ug_0\rightarrow \Uh$ is called the {\it Harish-Chandra
homomorphism} relative to $\frb$. The map
$$\psi:=\gamma^\sharp\circ\varphi|_{Z(\frg)}: Z(\frg)\car S(\frh)^W$$
is an algebra isomorphism independent of the choice of $\frb$.

\begin{rem} There is an important extension of this
construction to the case of a general parabolic subalgebra
$\frp\subset\frg$. This {\it generalized Harish-Chandra
homomorphism} is due to Drozd-Ovsienko-Futorny (\cite{DOF}) and is
a central tool in the study of generalized Verma modules. Since we
will not make use of it here (but see proof of thm.
\ref{propertiesparabolic} above) we refer to \cite{Mazorchuk}, 4.3
for a detailed description.
\end{rem}

\vskip8pt

It will be convenient to extend the above isomorphism $\psi$ to
Arens-Michael envelopes. That this is possible follows from work
of J. Kohlhaase on the center of $p$-adic distribution algebras
(\cite{KohlhaaseI}). We summarize the relevant results.
\begin{prop}\label{jan}
The $\frg$-action on $\Sg$ and $\Ug$ extends to Arens-Michael
envelopes and the same holds for the Weyl action on $\Sh$. The
algebra of invariants $\hUg^{\frg}$ coincides with the center
$\hZg$ of $\hUg$ and equals the closure of $Z(\frg)$. The
homomorphism $\psi$ extends to a topological isomorphism of
$K$-Fr\'echet algebras $$\hat{\psi}: \hZg\car \hSh^{W}.$$
\end{prop}
\begin{proof}
All this is contained in \cite{KohlhaaseI}, sect. 2.1. For example
the last statement follows from Prop. 2.1.5 and (proof of) Thm.
2.1.6.
\end{proof}
\begin{rem}\label{describecenter}
Let $X=\bbA^{l,an}_K$. The basis $\frH=\{h_1,..,h_l\}$ induces an
isomorphism $\frH:\hSh\car\cO(X)$ (cf. (\ref{basis})). It follows
that $W$ acts on $X$ by rigid analytic automorphisms. The
rigid-analytic quotient $X/W$ exists by finiteness of $W$
according to general principles ([loc.cit.]. The projection
$$X\longrightarrow X/W$$ is a finite morphism (\cite{BGR}, 9.4.4)
and as such has finite fibres (\cite{BGR} Cor. 9.6.3/6).
Alltogether $\frH$ induces a topological isomorphism
$$\overline{\frH}: \hSh^W\car\cO(X/W).$$ Finally, this situation is
the analytification of an algebraic action on algebraic affine
space via the finite group $W$. The usual description of
$S(\frh)^W$ as $l$-dimensional polynomial ring over $K$
(\cite{Dixmier}) therefore extends to completions yielding an
isomorphim of $K$-Fr\'echet algebras
$$\cO(X/W)\car\cO(\bbA^{l,an}_K)$$
onto the algebra of holomorphic functions on affine $l$-space. The
above proposition gives thus a very explicit description of the
center of $\hUg$. For more details we refer to \cite{KohlhaaseI}.
\end{rem}

\subsection{Central characters}

Recall that the usual {\it dot action} of $W$ on $\frh^*$ is given
by $w\cdot\lambda=w(\lambda+\rho)-\rho$ for $\lambda\in\frh^*,w\in
W$. Since translating the origin of $X=\bbA^{l,an}_K$ to $-\rho$
is a rigid isomorphism, say $\gamma$, the action extends to a
dot-action of $W$ on $X$ giving $\bar{\gamma}: X/W\car
X/(W,\cdot)$. Invoking Prop. \ref{jan} the composite
$(\bar{\gamma}^\sharp)^{-1}\circ\overline{\frH}\circ\hat{\psi}$ is
a canonical topological isomorphism
\begin{equation}\label{untwistedHC}
\hZg\car\hSh^W\car\cO(X/W)\car \cO(X/(W,\cdot))\end{equation} of
$K$-Fr\'echet algebras.

\vskip8pt

Now let $\lambda\in\frh^*$ and choose an irreducible highest weight module
$M\in\hO$ with maximal vector $m\in M_\lambda$. By Prop.
\ref{propobjO} we have $\End_{\hUg}(M)=K$ whence a continuous
character $\chi_\lambda: \hZg\longrightarrow K$. Since
$\hat{\psi}$ extends the $\gamma^\sharp$-twisted Harish-Chandra
homomorphism and since $Z(\frg)\subseteq \hZg$ is dense the
resulting map $\lambda\mapsto\chi_\lambda$ is induced by the rigid
analytic quotient morphism
$$\pi: X\longrightarrow X/(W,\cdot).$$
In particular, any continuous character $\chi:\hZg\rightarrow K$
arises, up to a finite extension of $K$, as some $\chi_\lambda$.


Moreover, a highest weight module of weight $\lambda$ has finite
length (Prop. \ref{propobjO}) and visibly all Jordan-H\"older
factors of such a module have highest weights contained in the
fibre of $\pi$ in $\chi=\pi(\lambda)$.

\vskip8pt

We turn back to the general case of a parabolic subalgebra
$\frp=\frp_I$ of $\frg$. We propose the following straightforward
variant of the classical decomposition of $\cO^\frp$ in terms of
central characters (\cite{HumphreysBGG},\cite{BGG2}). Let
$M\in\hO^\frp$ and let $\chi:\hZg\rightarrow K$ be a central
character. Then $\hZg$ acts on the weight spaces $M_\lambda$ ($\lambda$ being a $\hUh$-weight) and
we may form the subspace
$$M_\lambda^\chi:=\{m\in M_\lambda:
(\ker\chi)^n.m=0\text{\;for\;some\;} n=n(m)\geq 1\}.$$ Since
$\oplus_\lambda M_\lambda^\chi$ is a $U(\frg)$-submodule of
$M^{ss}$ its closure $M^\chi$ in $M$ is a subobject of $M$ (Lem.
\ref{propobjO}). We define the following full subcategory of
$\hO^\frp$: $$\hO^\frp_\chi:=\{M\in\hO^\frp: M^\chi=M\}.$$ In case
$I=\emptyset$ we write $\hO_\chi:=\hO_\chi^\frb$.
\begin{prop}\label{block}
The category $\hO_\chi^\frp$ is abelian.  The functor
$$\hO^\frp\rightarrow\hO^\frp_\chi, \;\;M\mapsto M^\chi$$ is exact and
induces an exact and faithful embedding of $\hO^\frp$ into the
direct product $\prod_\chi \hO^\frp_\chi$ (where $\chi$ runs
through the $K$-valued central characters).
\end{prop}
\begin{proof}
Since the inclusion $\hO^\frp\subseteq\hO$ is defined solely in
terms of weights (Def. \ref{def}) we are easily reduced to the
case $I=\emptyset$. Giving $\hO_\chi$ the exact structure coming
from $\hO$ let us show that $M\mapsto M^\chi$ is an exact functor.
Given a morphism $M\rightarrow N$ in $\hO$ we certainly have maps
$M_\lambda\rightarrow N_\lambda$ and $M_\lambda^\chi\rightarrow
N_\lambda^\chi$ for every $\chi$. Taking the sum over all
$\lambda$ and passing to closures with respect to the induced
subspace topologies we see that $M\mapsto M^\chi$ is indeed
functorial. Using strictness of maps in $\hO$ with respect to
canonical topologies (Prop. \ref{coadmissibles} (vii)) the same
argument yields its exactness (\cite{BGR}, Cor. 1.1.9/6). It is
now clear that the subcategory $\hO_\chi$ is closed under passage
to kernels and cokernels and, thus, abelian.

Now choose topological generators $z_1,...,z_l$ of $\hZg$
according to remark \ref{describecenter}. Then $M_\lambda^\chi$
equals the simultaneous generalized eigenspace of the finitely
many commuting operators $z_1,...,z_l$ on the finite dimensional
space $M_\lambda$ corresponding to the ordered set of eigenvalues
$\chi(z_1),...,\chi(z_l)$. In particular, there exists a finite
field extension $K\subseteq K'$ of $K$ such that $$K'\otimes_K
M_\lambda=\oplus_{\chi'} (K'\otimes_K M_\lambda)^{\chi'}$$ where
the sum runs over all $K'$-valued central characters $\chi'$ and
$(K'\otimes_K M_\lambda)^{\chi'}$ is defined in the obvious way.
We claim that $$(K'\otimes_K M_\lambda)^{\chi'}\neq 0\Rightarrow
\chi'(\hZg)\subseteq K.$$ Indeed, let $m\in K'\otimes_K M_\lambda$
and let $n\geq 1$ be minimal such that $(\ker\chi')^n.m=0$. On a
nonzero $m'\in (\ker\chi')^{n-1}.m$ the center $\hZg$ operates via
$\chi'$ and hence $\pi(\lambda)=\chi'$. In particular, $\chi'$ is
a $K$-valued point of $X/(W,\cdot)$ which proves the claim.

We therefore have $M_\lambda=\oplus_\chi M_\lambda^\chi$ with
$\chi$ running through the $K$-valued characters of $\hZg$.
Together with the obvious equality $M^{ss}\cap
M^\chi=\oplus_\lambda M_\lambda^\chi$ this implies
$M^{ss}=\oplus_\chi (M^\chi\cap M^{ss})$. It follows from this and
properties of $(\cdot)^{ss}$ that the sum $\sum_\chi M^\chi$ is
dense and direct in $M$. 
 In particular, the functor
$\hO\rightarrow\prod_\chi \hO_\chi, M\mapsto
(M^\chi)_\chi$ is faithful.

\end{proof}

\begin{prop}\label{finite}
The categories $\hOcp$ are artinian and noetherian.
\end{prop}
\begin{proof}
With the $p$-adic Harish Chandra map at hand we may imitate the
classical argument (\cite{BGG2},\cite{Dixmier}) as follows. Since
$\hOcp\subseteq\hOc$ is a full subcategory we may suppose
$I=\emptyset$. Let $M\in\hOc$ be given and put
$V:=\sum_{\mu\in\pi^{-1}(\chi)} M_{\mu}$. Since $\pi$ has finite
fibers we have $\dim_KV<\infty$. Suppose $N'\varsubsetneq
N\subseteq M$ are two subobjects. Let $m\in N/N'$ be a maximal
vector of some weight $\mu$. Since the subobject $\hUg.m\subseteq
N/N'$ is a highest weight module $\hZg$ operates on $m$ via
$\chi_\mu$. Hence $\chi_\mu=\chi$ and $\mu\in \pi^{-1}(\chi)$. By
definition $m\in N\cap V$ whence $\dim_K N\cap V>\dim_K N'\cap V$.
This shows $M$ to be artinian and noetherian.
\end{proof}

\subsection{The main result}

In this section we prove the following main result. As with any
Arens-Michael envelope (compare \ref{subsecthyper}) we have a
natural map
$$U(\frg)\rightarrow \hat{U}(\frg).$$

\begin{thm}\label{mainresult}
The functor $M\mapsto \hat{U}(\frg)\otimes_{U(\frg)} M$ induces an
equivalence of categories
$$\cO^\frp\car\hO^\frp.$$
A quasi-inverse is given by $(\cdot)^{ss}$. The equivalence
identifies $\cO^\frp_\chi\simeq\hO^\frp_\chi$ for any $K$-valued
central character $\chi$ and hence
$\hat{\mathcal{O}}^{\mathfrak{p}} = \prod_\chi
\hat{\mathcal{O}}^{\mathfrak{p}}_\chi $.
\end{thm}
According to well-known properties of $\cO^\frp$ (cf.
\cite{HumphreysBGG}, Thm. 9.8,\cite{Mazorchuk}, 5.2) we obtain
\begin{cor}
The category $\hO^\frp$ has enough injectives and projectives and
a duality. Each block $\hO^\frp_\chi$ is a highest weight category
and (noncanonically) equivalent to a category of finitely
generated right modules over a BGG algebra.
\end{cor}

To begin the proof of the theorem let us first assume
$I=\emptyset$. Recall (example \ref{finite0}) the fully faithful
embedding from the finite dimensional $\frg$-modules into $\hO$.
Any finitely generated (left) $U(\frg)$-module is finitely
presented and therefore $M\mapsto \hUg\otimes_{\Ug}M$ constitutes
a functor $F$ from such modules into $\cC_\frg$ (Prop.
\ref{coadmissibles} (v)). Our further investigation relies on the
following fact.
\begin{thm}
The extension $\Ug\rightarrow\hUg$ is flat.
\end{thm}
\begin{proof}
This is a direct consequence of the main result of
\cite{SchmidtSTAB}. Alternatively, one may pass to a finite
extension of $K$ and use flatness of adic completion at central
ideals of noetherian rings. This yields the flatness of the map
$U(\mathfrak{g})\rightarrow U_r(\mathfrak{g})$ for $r\in
p^{\mathbb{Q}}$ and then [28], (proof of) Thm. 4.11 gives the
claim.
\end{proof}
In particular, $F$ is exact. It is almost obvious that
$F(M(\lambda))=\hMl$ and hence any highest weight module of $\cO$
is mapped to a highest weight module in $\hO$. Since any module
$M$ in $\cO$ has a finite filtration with graded quotients being
highest weight modules (\cite{HumphreysBGG}, Cor. 1.2)
there is a surjection $\oplus_i
M_i\stackrel{\Sigma}{\longrightarrow} M$ where the source is a
finite direct sum of highest weight modules. Since $F$ commutes
with direct sums we see $F(M)\in\hO$. We have thus established an
exact functor
$$F:\cO\longrightarrow \hO$$
extending the aforementioned embedding of the finite dimensional
modules into $\hO$.
\begin{prop}\label{leftquasi}
The functor $F$ is fully faithful. A left quasi-inverse is given
by $(\cdot)^{ss}$.
\end{prop}
\begin{proof}
If $M\in\cO$ and $m\in M_\lambda$ the map $m\mapsto 1\otimes m$
induces a $U(\frh)$-linear homomorphism from $M_\lambda$ into the
$\lambda$-weight space of $F(M)^{ss}$. It extends to a
$U(\frg)$-linear homomorphism $M\rightarrow F(M)^{ss}$ natural in
$M$. If $M$ is a Verma module it is bijective according to Prop.
\ref{verma}. If $M\in\cO$ is a highest weight module we consider
an exact sequence
$$0\longrightarrow N\longrightarrow M(\lambda)\longrightarrow
M\longrightarrow 0$$ for suitable $\lambda\in\frh^*$. Writing $N$
as a subquotient of the left regular module $U(\frg)$ and
recalling (proof of Prop. \ref{verma}) that $\hMl$ equals the
completion of $U(\frg)/U(\frg)J=M(\lambda)$ with respect to the
(separated) quotient topology one sees that the natural injection
$F(N)\rightarrow \hMl$ has image equal to the closure of $N$ in
$\hMl$.
Hence, $N \simeq F(N)^{ss}$ by Prop. \ref{diag}(iii) and therefore
$M\simeq F(M)^{ss}$. Now let $M\in\cO$ be arbitrary. By devissage
we may assume that $M$ is an extension of highest weight modules.
But then $M\simeq F(M)^{ss}$ by the Five lemma.
\end{proof}

Next we will fix a number $r>1$ in $p^\bbQ$ and consider the
noetherian Banach algebra $U_r(\frg)$ (section
\ref{subsecthyper}). Working with our standard basis (section
\ref{subsectsemisimple}) of $\frg$ we see that the inclusion
$U(\frh)\subseteq U(\frg)$ extends to an isometry
$U_r(\frh)\subseteq U_r(\frg).$ Similarly we obtain isometries
$U_r(\frn),U_r(\frn^-)\subseteq U_r(\frg)$. We denote by
$$U_r(\frn^-)\hat{\otimes}_K U_r(\frh)\hat{\otimes}_K U_r(\frn)$$
the completion of the tensor product of these subalgebras with
respect to the usual tensor product norm (which coincides with the
completed projective tensor product, \cite{NFA}, Lem. 17.2).
\begin{lemma}
There PBW-decomposition of Lem. \ref{PBW} extends to an isometry
of Banach $(U_r(\frn^-),U_r(\frn))$-bimodules
$$U_r(\frn^-)\hat{\otimes}_K U_r(\frh)\hat{\otimes}_K
U_r(\frn)\car U_r(\frg).$$
\end{lemma}
\begin{proof}
The algebra structure being irrelevant here we may replace $U$ by
$S$. Since $r\in p^\bbQ$ passing to a finite extension of $K$
reduces us, by faithfully flat descent, to the case $r\in
|K^\times|$. We may therefore assume $r=1$ in which case the
result is well-known (\cite{BGR}, Cor. 6.1.1/8).
\end{proof}

Given a coadmissible module $M$ we let
$$M_r:=U_r(\frg)\otimes_{U(\frg)} M.$$ By the general
Fr\'echet-Stein formalism $M_r$ is a finitely generated Banach
$U_r(\frg)$-module and the natural map $M\rightarrow M_r$ has
dense image (\cite{ST5}, \S3). The following lemma is due to
Benjamin Schraen and I thank him for allowing me to reproduce it
here.

\begin{lemma} We have $\hLl_r\neq 0$ for any weight
$\lambda\in\frh^*$.
\end{lemma}
\begin{proof}
Consider the kernel $Q$ of the natural map $\hMl\rightarrow\hLl$.
Since $U(\frg)\rightarrow U_r(\frg)$ is flat the kernel of
$\hMl_r\rightarrow \hLl_r$ equals $Q_r$. Applying the above lemma
we see that $$\hMl_r\simeq U_r(\frn^-)\otimes_K K_\lambda$$ whence
$\hMl_r$ is $U_r(\frh)$-diagonalisable with
$$\hMl^{ss}=(\hMl_r)^{ss}$$ via the inclusion $\hMl\subseteq\hMl_r$.
By Prop. \ref{exact} the modules $Q_r$ and $\hLl_r$ are
$U_r(\frh)$-diagonalisable and it suffices to see that
$$(Q_r)^{ss}\subsetneq (\hMl_r)^{ss}.$$ Now $(Q_r)^{ss}\subseteq
Q_r$ is dense. Similarly, the composite map $Q^{ss}\subseteq
Q\subseteq Q_r$ has dense image. According to Prop. \ref{diag} we
therefore obtain $Q^{ss}=(Q_r)^{ss}$ as abstract
$U_r(\frh)$-submodules of $Q_r$. Since $\hLl\neq 0$ we arrive at
$$(Q_r)^{ss}=Q^{ss}\subsetneq\hMl^{ss}=(\hMl_r)^{ss}.$$
\end{proof}
\begin{lemma}\label{finite2}
The category $\hO$ is artinian and noetherian.
\end{lemma}
\begin{proof}
Let $M\in\hO$. Recall that $\oplus_\chi M^\chi$ is dense in $M$
according to Prop. \ref{block}. Letting $r>1$ in $p^\bbQ$ we
compute
$$M_r=U_r(\frg)\otimes_{U(\frg)}M\supseteq U_r(\frg)\otimes_{U(\frg)}(\oplus_\chi
M^\chi)=\oplus_\chi (M^\chi)_r.$$ Any nonzero $M^\chi$ has a
composition series (Prop. \ref{finite}) whence $\hLl_r\subseteq
(M^\chi)_r$ for some weight $\lambda\in\frh^*$ and then
$(M^\chi)_r\neq 0$ by the preceding lemma. Since $M_r$ is finitely
generated and $U_r(\frg)$ is noetherian this means that $M^\chi=0$
for all but finitely many $\chi$. But then $\oplus_\chi M^\chi$ is closed in $M$ according to Prop. \ref{coadmissibles}.
\end{proof}
\begin{lemma}
Given $M\in\hO$ the abstract $U(\frg)$-module $M^{ss}$ lies in
$\cO$. The correspondence $M\mapsto M^{ss}$ is a quasi-inverse to
$F$.
\end{lemma}
\begin{proof}
By the preceding result we may assume that $M$ is an extension of
two simple objects. According to the result (\ref{simpless}) we
see that $M^{ss}$ is a finitely generated $U(\frg)$-module on
which $\frh$ acts semisimple. By our assumption on the weights
$\Pi(M)$ the algebra $\frn$ acts locally finite. This means
$M^{ss}\in\cO$. To prove the second statement it suffices,
according to Prop. \ref{leftquasi}, to show that $(\cdot)^{ss}$ is
right quasi-inverse to $F$. Let $M\in\hO$. If we apply
$\hat{U}(\frg)\otimes_{U(\frg)} (\cdot)$ to the inclusion
$M^{ss}\subseteq M$ and compose with the map $u\otimes m\mapsto
um$ we obtain a morphism $F(M^{ss})\rightarrow M$ in $\hO$. If $K$
and $Q$ denote its kernel and cokernel respectively we have
$K^{ss}=Q^{ss}=0$ by Prop. 4.3.4 whence $K=Q=0$ by Prop.
\ref{diag}.
\end{proof}
This ends the proof of the theorem in case $I=\emptyset$. Now
consider the case of a general parabolic subalgebra $\frp=\frp_I$.
Since the functors $F$ and $(\cdot)^{ss}$ preserve $\frh$-weight
spaces Lem. \ref{char} shows that our established equivalence
$\hO\simeq\cO$ identifies the full subcategories
$\cO^\frp\subseteq\cO$ and $\hO^\frp\subseteq\hO$. It is obvious
that this identification respects the central blocks. This
finishes the proof of the theorem.

\appendix

\section{Quasi-hereditary algebras and highest weight
categories}\label{sectquasi}

We recall some basic facts on quasi-hereditary algebras and
highest weight categories. The following formulation is adapted to
our purposes. For more details we refer to \cite{DlabRingel} and
\cite{Farnsteiner}.

\bigskip

Let $K$ be a field, $A$ a finite dimensional $K$-algebra,
$\Mod_{\rm fg}(A)$ the category of finitely generated right
$A$-modules and $K_0(A)$ the Grothendieck group of $\Mod_{\rm
fg}(A)$. Let
$$(\Lambda,\leq)$$ be a fixed partially ordered finite set indexing a full set of representatives
$(L_\lambda)_{\lambda\in\Lambda}$ for the isomorphism classes of
simple right $A$-modules. The multiplicity of $L_\lambda$ in a
Jordan-H\"older series of a module $M$ will be denoted by
$[M:L_\lambda]$. Given $\lambda\in\Lambda$ let $P_\lambda$ and
$I_\lambda$ be a projective cover and injective hull of
$L_\lambda$ in $\Mod_{\rm fg}(A)$ respectively.

\bigskip

A collection of {\it standard modules} for $A$ (relative to the
partially ordered set $\Lambda$) is a set $\Delta$ of modules
$\Delta_\lambda\in\Mod_{\rm fg}(A)$ with the properties
$[\Delta_\lambda:L_\lambda]=1$ with ${\rm
Top}(\Delta_\lambda)\simeq L_\lambda$ and
$[\Delta_\lambda:L_\mu]=0$ if $\mu\nleq\lambda$.
Given such a set $\Delta$ let $\cF(\Delta)$ be the full
subcategory of $\Mod_{\rm fg}(A)$ consisting of modules $M$
admitting a finite filtration with graded quotients isomorphic to
members of $\Delta$. Given $M\in\cF(\Delta)$ the element $[M]$ of
$K_0(A)$ can be written as
$$[M]=\sum_{\lambda\in\Lambda}n_\lambda[\Delta_\lambda]=\sum_{\lambda\in\Lambda}
n_\lambda\sum_{\mu\in\Lambda} [\Delta_\lambda:L_\mu][L_\mu]$$ with
suitable $n_\lambda\in\bbN$. Choose a numbering
$\lambda_1,...,\lambda_s$ of the elements in $\Lambda$ such that
$\lambda_i <\lambda_j$ implies $i>j$. The matrix
$([\Delta_\lambda:L_\mu])_{\lambda,\mu}$ is then unipotent upper
triangular and since the elements $[L_\mu]$ form a $\bbZ$-basis of
$K_0(A)$, the coefficients $n_\lambda$ are uniquely determined.
The filtration multiplicities $(M:\Delta_\lambda)$ are therefore
independent of the choice of filtration. Finally, the standard
module $\Delta_\lambda$ is called {\it schurian} if
$\End_A(\Delta_\lambda)$ is a division ring.

\bigskip

Recall that in this situation $A$ is called (right) {\it
quasi-hereditary} if all standard modules are schurian and we have
$P_\mu\in \cF(\Delta)$ such that $(P_\mu:\Delta_\mu)=1$ and
$(P_\mu:\Delta_\lambda)=0$ if $\mu\nleq\lambda$ for all
$\lambda,\mu\in\Lambda$ (cf. \cite{DlabRingel}, \S1).

\begin{rem}
Let $A$ be quasi-hereditary with set of standard modules $\Delta$.
If $\leq_1$ is a total ordering on $\Lambda$ that contains $\leq$
then, trivially, $(A,\leq_1)$ is quasi-hereditary with the same
set of standard modules. In dealing with quasi-hereditary algebras
we may therefore always assume that $\Lambda=\{1,...,n\}$, some
$n$, equipped with its natural ordering. In other words, the issue
of a {\it non-adapted} $\Lambda$ (in the sense of [loc.cit.]) does
not arise here.
\end{rem}
\begin{rem}\label{highestweight} Let $A$ be quasi-hereditary.
Without recalling a precise definition we remark that $\Mod_{\rm
fg}(A)$ is a highest weight category in the sense of
Cline-Parshall-Scott (cf. \cite{ClineParshallScott}, Lem. 3.4).
\end{rem}
If $A$ is a quasi-hereditary algebra it is easy to see that each
$I_\lambda$ has a unique largest submodule $\nabla_\lambda$ with
$[\nabla_\lambda:L_\mu]=0$ for $\mu\nleq\lambda$. The modules
$\nabla_\lambda$ are sometimes called the {\it costandard modules}
associated to $A$ (cf. \cite{DlabRingel}, \cite{Farnsteiner}).
\begin{prop}
Let $A$ be quasi-hereditary. Then $A$ has (right) global dimension
bounded by $2|\Lambda|$. \end{prop}
\begin{proof}
This follows from \cite{DlabRingel}, Lem. 2.2.
\end{proof}
\begin{rem} If the (right) global dimension of a finite dimensional $K$-algebra $A$ is $\leq 1$ then $A$ is (right) hereditary, i.e.
 all right ideals are projective. (\cite{AuslanderReitenbook}, Cor. 5.2). For the extensive and well-understood theory of hereditary algebras we refer to [loc.cit.], chap. VIII.
\end{rem}

\begin{prop}\label{weakBGG}
Let $A$ be a quasi-hereditary algebra. Then
$$(P_\mu:\Delta_\lambda)\cdot d_\lambda=[\nabla_\lambda:L_\mu]\cdot d_\mu$$
where $d_\lambda:=\dim_K \End_A(\Delta_\lambda)$ for all
$\lambda,\mu\in\Lambda$.
\end{prop}
\begin{proof}
\cite{DlabRingel}, Lem. 2.5. and \cite{Farnsteiner}, Thm. 3
\end{proof}
A quasi-hereditary algebra is called a {\it BGG-algebra} if there
exists a contravariant involutive autofunctor $D$ on $\Mod_{\rm
fg}(A)$ such that $D(L_\lambda)\simeq L_\lambda$ for all
$\lambda\in\Lambda$ (\cite{Irving}). Such an algebra satisfies the
so-called {\it strong BGG reciprocity}:
\begin{prop}\label{strongBGG}
Let $A$ be a BGG algebra. Then
$$(P_\mu:\Delta_\lambda)\cdot d_\lambda=[\Delta_\lambda:L_\mu]\cdot d_\mu$$
where $d_\lambda:=\dim_K \End_A(\Delta_\lambda)$ for all
$\lambda,\mu\in\Lambda$.
\end{prop}
\begin{proof}
It is easy to see that $D(\nabla_\lambda)\simeq\Delta_\lambda$ for
all $\lambda$ (\cite{Farnsteiner}, Lem. 4). The claim follows thus
from the above proposition using that $D$ preserves
Jordan-H\"older multiplicities.
\end{proof}
\begin{example}
Let $K$ be a $p$-adic local field, $\frg$ a split reductive Lie
algebra over $K$, $\frb$ a Borel subalgebra and
$\frp\subseteq\frg$ a parabolic subalgebra containing $\frb$.
Denote by $\cO^\frp$ the parabolic {\it BGG category} of $\frg$
relative to $\frp$ (cf. \ref{parabolicBGG}). Let $\chi$ be a
$K$-valued character of $Z(\frg)$ and $\cO^\frp_\chi$ the
corresponding central block of $\cO^\frp$. Then $\cO^\frp_\chi$ is
(noncanonically) equivalent to the category of finitely generated
(right) modules over a BGG-algebra $A^\frp_\chi$. A set of
schurian standard modules is given by the GVM's $M_I(\lambda)$
contained in $\cO^\frp_\chi$ (where $\frp=\frp_I$). The algebra
$A^\frp_\chi$ arises (a direct consequence of the theorem of
Gabriel- Mitchell (\cite{Bass}, Thm. II.1.3 and subsequent
exercise)) as the endomorphism algebra of a suitable projective
generator of the artinian and noetherian category $\cO^\frp_\chi$.
By work of W. Soergel its structure-at least in the case
$\frp=\frb$- can be explicitly determined (\cite{Soergel}).

\vskip8pt

In case of a complex semisimple algebra and a Borel subalgebra
this is the urexample in the theory of quasi-hereditary algebras
(cf. \cite{BGG2} and \cite{ClineParshallScott}, Example 3.3 (c)).
\end{example}

\bibliographystyle{plain}
\bibliography{mybib}

\end{document}